\documentclass[10pt]{article}

\usepackage{amsfonts}
\usepackage{amsmath}
\usepackage{amssymb}
\usepackage{amsthm}
\usepackage{ifthen}

\newtheorem{theorem}{Theorem}[section]

\newtheorem{corollary}[theorem]{Corollary}
\newtheorem{lemma}[theorem]{Lemma}
\newtheorem{problem}[theorem]{Problem}
\theoremstyle{remark}

\newtheorem{example}[theorem]{Example}
\theoremstyle{definition}
\newtheorem{definition}[theorem]{Definition}

\newcommand{\parder}[3][Default]{
	\frac{\partial \ifthenelse{\equal{#1}{Default}}{}{^{#1}}#2}{
              \partial #3 \ifthenelse{\equal{#1}{Default}}{}{^{#1}}}}

\newcommand{\jac}{{\mathcal J}}
\newcommand{\hess}{{\mathcal H}}
\newcommand{\grad}{\nabla}

\newcommand{\GL}{\operatorname{GL}}
\newcommand{\GO}{\operatorname{GO}}

\newcommand{\C}{{\mathbb C}}
\newcommand{\F}{{\mathbb F}}
\newcommand{\R}{{\mathbb R}}
\newcommand{\Q}{{\mathbb Q}}

\newcommand{\Z}{{\mathbb Z}}
\newcommand{\tp}{^{\rm t}}

\newcommand{\I}{{\mathrm i}}

\title{Polynomials with constant Hessian determinants in dimension three}
\author{Michiel de Bondt\footnote{The author was supported by the Netherlands 
        Organisation for Scientific Research (NWO).} \\
        Institute for Mathematics, Astrophisics and Particle Physics \\
        Radboud University Nijmegen \\
        \emph{Email address:} M.deBondt@math.ru.nl}

\begin{document}

\maketitle

\begin{abstract}
\noindent
In this paper, we show that the Jacobian conjecture holds for gradient maps
in dimension $n \le 3$ over a field $K$ of characteristic zero. We do this by 
extending the following result for $n \le 2$ by F. Dillen to $n \le 3$:
if $f$ is a polynomial of degree larger than two in $n \le 3$ variables such that the 
Hessian determinant of $f$ is constant, then after a suitable linear transformation 
(replacing $f$ by $f(Tx)$ for some $T \in \GL_n(K)$), the Hessian 
matrix of $f$ becomes zero below the anti-diagonal. The result does not hold for larger $n$.

The proof of the case $\det \hess f \in K^{*}$ is based on the following result, 
which in turn is based on the already known case $\det \hess f = 0$: 
if $f$ is a polynomial in $n \le 3$ variables such that $\det \hess f \ne 0$, 
then after a suitable linear transformation, there exists a positive weight function $w$ 
on the variables such that the Hessian determinant of the $w$-leading part of $f$ is nonzero. 
This result does not hold for larger $n$ either (even if we replace `positive' by `nontrivial'
above).

In the last section, we show that the Jacobian conjecture holds for
gradient maps over the reals whose linear part is the identity map, by proving that 
such gradient maps are translations (i.e.\@ have degree $1$) if they satisfy 
the Keller condition. We do this by showing that this problem is the polynomial case of 
the main result of \cite{MR0319126}. For polynomials in dimension $n \le 3$, we generalize 
this result to arbitrary fields of characteristic zero.
\end{abstract}

\noindent
\emph{Key words:} Hessian determinant, Jacobian conjecture, polynomial, weighted degree,
anisotropic quadratic form.

\medskip
\noindent 
\emph{2010 Mathematics Subject Classification:} 13F20; 14R15; 15A63.

\section{Introduction}

Throughout this paper, $K$ denotes an arbitrary field of characteristic zero.
Furthermore, $x_1, x_2, x_3, \ldots$ are indeterminates and $\GL_n(K)$ is the 
group of invertible matrices of size $n \times n$ over $K$. Applying a 
\emph{linear transformation} on a polynomial $f \in K[x] = K[x_1,x_2,\ldots,x_n]$ is
just replacing $f$ by $f(Tx)$, where $Tx$ can be seen as a matrix product of the matrix 
$T \in \GL_n(K)$ and the vector $x = (x_1,x_2,\ldots,x_n)$.

Suppose that the Hessian matrix $\hess f$ of a polynomial $f \in K[x] = K[x_1,x_2,\allowbreak
\ldots,x_n]$ is zero below the anti-diagonal. Then $\hess f$ can be turned into a lower triangular
matrix by way of $\lfloor n/2 \rfloor$ row interchanges. 
Since the product of the diagonal entries of a lower triangular matrix is just its determinant,
we can deduce that the product of the elements of the anti-diagonal of $\hess f$
is $(-1)^{\lfloor n/2 \rfloor} \det \hess f$.

But in general, the Hessian of a polynomial will not be zero below the anti-diagonal.
However, if $\det \hess f \in K$ and $n \le 3 \le \deg f$, then after applying a suitable linear 
transformation on $f$, $\hess f$ will indeed be zero below the anti-diagonal, where $\deg f$ 
denotes the total degree of $f$ (in $x_1,x_2,\ldots,x_n$). This is our main result. 
As a consequence, we show that the Jacobian conjecture holds for polynomial maps over $K$ 
in dimension $n \le 3$, for which the Jacobian is symmetric. Both our main result and its 
consequence can be found in theorem \ref{dillenext}. Section \ref{first} will be devoted to 
the proof of theorem \ref{dillenext}.

Suppose next that $\det \hess f \ne 0$ for some polynomial $f \in K[x] = K[x_1,x_2,
\allowbreak \ldots,x_n]$.
Now take a hyperplane $S$ with negative slope with respect to all coordinates of $\R^n$,
which intersects the newton polytope of $f$ only at its boundary. Let $\bar{f}$ be the
part of $f$ consisting of the monomials whose multidegrees lie in $S$. Although $\det \hess f \ne 0$,
it is very possible that $\det \hess \bar{f} = 0$. However, if $n \le 3$, then after applying a 
suitable linear transformation on $f$, $\det \hess \bar{f}$ will indeed be nonzero for some 
hyperplane $S$ as above. This is our main lemma for the proof of our main result. But since it is 
interesting on its own, this lemma has become a theorem, namely theorem \ref{weight}. 

Example \ref{weightcounter} makes clear that theorem \ref{weight} is no longer true in dimensions
larger than three. In fact, in dimension four and up, it is possible that $f$ has the following 
property, in such a way that it cannot be undone by applying linear transformations: the  
property of $f$ that $\det \hess \bar{f} \ne 0$ is entirely encapsulated by the Newton polytope of $f$.
More precisely, the part of $f$ consisting of monomials whose supports lie on the boundary of
the Newton polytope of $f$ is a polynomial whose Hessian determinant is zero.

The rest of this introduction will be devoted to a historical overview of the study of
polynomials with constant Hessian determinants, followed by a closer look on linear transformations.

There are some very old papers devoted to the study of polynomials with constant Hessian determinants 
in some manner. Perhaps the oldest is an article of Paul Gordan and Max N{\"o}ther about homogeneous
polynomials with Hessian determinant zero, which appeared in 1876 in \cite{MR1509898}. 
This is the most interesting case for homogeneous 
polynomials, because if a homogeneous polynomial $h \in \C[x] = \C[x_1,x_2,\ldots,x_n]$ has a constant 
nonzero Hessian determinant, then the Hessian matrix must be constant and nonzero, so that $h$ can only
be a quadratic form.

For quadratic forms, basic linear algebra can be used 
to show that $h$ can be written as a polynomial in $n-1$ (or less) linear forms over $\C$, 
if and only if the Hessian determinant of $h$ is zero. So assume that $h$ is a homogeneous
polynomial of degree $d \ge 3$. Again by basic linear algebra, it follows that $h$ cannot be 
written as a polynomial in  $n-1$ (or less) linear forms over $\C$, in case the Hessian 
determinant of $h$ is nonzero.

But the converse may not be true. However, in \cite{MR1509898},
the authors show that $h$ can indeed be written as a polynomial in $n-1$ (or less) linear forms 
over $\C$ in case $n \le 4$ and $h$ has Hessian determinant zero, 
and give counterexamples for all $n \ge 5$ and all $d \ge 3$.
In \cite{MR2095579}, A. van den Essen and the author classify all (not necessarily homogeneous) 
polynomials $h \in K[x_1,x_2,\ldots,x_n]$ with $n \le 3$, such that the Hessian determinant of $h$ is zero,
where $K$ is a field of characteristic zero, using techniques of \cite{MR1509898}. 
We shall use these results to prove our main lemma (theorem \ref{weight}) and the 
case $\det \hess f = 0$ of our main theorem (theorem \ref{dillenext}).

In 1939 in \cite{MR1550818}, O. Keller formulated a question about constant nonzero Jacobian determinants, 
which is known as Keller's problem or the Jacobian conjecture. The Jacobian conjecture asserts that 
$F$ is invertible if $F$ satisfies the so-called \emph{Keller condition}. The Keller condition
on a polynomial map $F$ is the property that $\det \jac F$ is a nonzero constant in $K$,
where $\jac F$ is the Jacobian matrix of $F$. 
Since Hessians are Jacobians of gradient maps, we can ask ourselves whether 
the Jacobian conjecture holds for gradient maps. This was done in \cite{MR2038568}, where a positive 
answer was given in dimension $n \le 4$, for gradient maps of the form $x + H$ with $H$ homogeneous.
See also \cite{MR2076378} for this result. In \cite{MR2110519}, A. van den Essen and the author
generalized this result to dimension $n \le 5$. For dimension $n \le 4$, the condition that 
$H$ is homogeneous was weakened to that $\jac H$ is nilpotent (the nilpotency of
$\jac H$ follows by way of the condition that $\det \jac F$ is a nonzero constant
from the homogeneity of $H$).

The generalization to dimension $n \le 5$ led the authors of \cite{MR2110519} to the discovery
that the Jacobian conjecture for gradient maps is equivalent to the Jacobian conjecture, 
which they published in \cite{MR2138860}, see also \cite[Cor.\@ 1.4]{MR3179983}. 
G. Meng did the same discovery, which he published as \cite[Prop.\@ 1.4]{MR2221506}, see also
\cite[Th.\@ 1.2]{MR3179983}.
More precisely, the Jacobian conjecture in dimension $n$ follows from the Jacobian conjecture for 
gradient maps in dimension $2n$. Hence the Jacobian conjecture is about polynomials with constant 
Hessian determinants after all. We shall show that the Jacobian conjecture holds for gradient maps 
in dimension three. In dimension two, this problem has already been solved in 1991 by F. Dillen,
see below. Notice that an affirmative answer to the same problem in dimension four would imply the 
planar Jacobian conjecture.

In \cite[Th.\@ 1.3]{MR2221506}, the author proves that the Jacobian conjecture holds for real gradient 
maps in all dimensions, provided the linear part is equal to the identity map. We will reprove
this result, by showing that real gradient maps whose linear part is the identity map are 
translations (i.e.\@ have degree $1$) if they satisfy the Keller condition, using results that are 
described below. 
In 1954, K. J{\"o}rgens proved in \cite{MR0062326} that functions from $\R^2$ to $\R$ which are 
twice continuously differentiable and whose Hessian determinant equals one at each point are in fact 
quadratic polynomials. Four years later, this result was extended to $\R^3$ and $\R^4$ by E. Calabi in 
\cite{MR0106487}, but with the extra condition that the Hessian matrix is positive definite 
everywhere. The polynomial 
$$
f = g(x_1 + x_3) - \tfrac12 x_1^2 - \tfrac12 x_2^2 + \tfrac12 x_3^2 + \cdots + \tfrac12 x_n^2 
$$
shows that such an extra condition is required. An extension to arbitrary dimension was proved by A.V. 
Pogorelov in 1972 in \cite{MR0319126}, using a lemma of \cite{MR0106487}. In theorem \ref{definite}, we 
shall extend this result for polynomials in dimensions $n \le 3$ (and for real polynomials) as follows. 
We shall show that a polynomial in $K[x_1,x_2,\ldots,x_n]$ is a quadratic polynomial if its Hessian 
determinant is constant and its quadratic part does not vanish at $K^n \setminus \{0\}^n$,
provided $n \le 3$ (or $K = \R$). 

In 1991, F. Dillen classified all polynomials in two indeterminates over a field of characteristic
zero with constant Hessian determinant in \cite{MR1107649}, and showed that the Jacobian conjecture
holds for gradient maps in dimension two. The key point of Dillen's classification can be seen
as follows: if the degree of the polynomial is larger than two, then after a suitable linear transformation, 
the lower right corner of its Hessian matrix becomes zero.
As mentioned above, our main result (in theorem \ref{dillenext}) is a similar statement for polynomials
with constant Hessian determinant in dimension three:
if the degree of the polynomial is larger than two, then after a suitable linear transformation, 
every entry below the anti-diagonal of its Hessian matrix becomes zero. 
This result does not hold in dimensions larger than three and neither for quadratic polynomials over 
$\R$ in dimensions two and three.

Let $T \in \GL_n(K)$. Taking the Jacobian matrix of $f(Tx)$, we obtain by the chain rule that
$$
\jac f(Tx) = (\jac f)|_{x=Tx} \cdot T
$$
where $|_{x=g}$ stands for substituting $x$ by $g$.
Since the gradient vector, denoted by $\grad$, is the transpose of the Jacobian of a single 
polynomial, we obtain
\begin{equation} \label{gradT}
\grad f(Tx) = T\tp \cdot (\grad f)|_{x=Tx}
\end{equation}
where $\tp$ stands for taking the transpose.
Subsequently, we can take the Jacobian of (\ref{gradT}), which is the Hessian, denoted by $\hess$,
of $f(Tx)$, and again by the chain rule, we obtain
\begin{equation} \label{hessT}
\hess f(Tx) = T\tp \cdot (\hess f)|_{x=Tx} \cdot T
\end{equation}
Formulas \ref{gradT} and \ref{hessT} indicate the effect of a linear transformation 
on the gradient and the Hessian, respectively.

Now that we know how linear transformations influence the Hessian, we are able to see that Dillen's result 
cannot be extended to dimension four. Take for instance 
\begin{equation} \label{dillen4}
f = (x_1 + x_2^2) x_3 + (x_2 + (x_1 + x_2^2)^2) x_4
\end{equation}
Then the cubic part of $f$ is equal to $x_2^2 x_3 + x_1^2 x_4$, and the rows of its Hessian, whose
entries are linear forms, are independent over $\C$. This is maintained after a linear transformation, 
so if we could obtain by way of a transformation that
the Hessian of $f$ becomes zero below the anti-diagonal, the lower left corner of $\hess f$ would 
get a nontrivial linear part, because the last row of the Hessian of the cubic part of 
$f$ cannot become zero. This however contradicts that the Hessian determinant is a 
nonzero constant. The polynomial $f$ in (\ref{dillen4}) was made by applying the reduction
of the Jacobian conjecture to gradient maps, as in \cite[Prop.\@ 1.4]{MR2221506}, on the planar 
invertible map $F = (x_1 + x_2^2, x_2 + (x_1 + x_2^2)^2)$.

\section{Results and proof of main result} \label{first}

First, we formulate our main result. At the end of this section, we will derive the main result 
from other results in this section.

\begin{theorem}[Main result] \label{dillenext}
Let $K$ be a field of characteristic zero and $f \in K[x] = K[x_1,x_2,\ldots,x_n]$ 
be a polynomial of degree $d$. If $n \le 3$, then $\grad f$ satisfies the Jacobian conjecture.

If $n \le 3 \le d$ and $\det \hess f \in K$, then there exists a $T \in \GL_n(K)$ such that all 
entries below the anti-diagonal of the Hessian of $f(Tx)$ are zero. In particular, the quadratic 
part of $f$ vanishes somewhere at $K^n \setminus \{0\}^n$ when $2 \le n \le 3$, namely at the last 
column of $T$.
\end{theorem}

\noindent
The condition $d \ge 3$ for the existence of $T$ as claimed in theorem \ref{dillenext} is 
necessary (except for algebraically closed fields $K$).
Take for instance $f = \frac12 x_1^2 + \frac12 x_2^2 + \cdots + \frac12 x_n^2$. Then $f$ has no
nontrivial zero over $\R$, so $f$ does not vanish at the last column of any $T \in \GL_n(\R)$.
See section \ref{last} for more results about the real numbers.

Let $K$ be a field. We call $w: K[x_1,x_2,\ldots,x_n] \rightarrow \R \cup \{-\infty\}$ 
a {\em weight function} if $w(0) = -\infty$,
$$
w(x_1^{\alpha_1}x_2^{\alpha_2}\cdots x_n^{\alpha_n}) = 
\alpha_1 w(x_1) + \alpha_2 w(x_2) + \cdots + \alpha_n w(x_n) \in \R
$$ 
and, in case $g \ne 0$,
$$
w(g) = \max\{w(t)\mid \mbox{$t$ is a term of $g$ (with nonzero coefficient)}\}
$$
The {\em $w$-leading part} of a polynomial $g$ is the sum of the monomials $c_t t$ of $g$ 
(with nonzero coefficient $c_t$), for which $w(t) = w(g)$.

In order to prove our main theorem, we use the following theorem, which we prove in section
\ref{second}.

\begin{theorem} \label{weight}
Let $K$ be a field of characteristic zero and assume that $f \in K[x] = K[x_1,x_2,\ldots,x_n]$
satisfies $\det \hess f \ne 0$. If $n \le 3$, then there exist a $T \in \GL_n(K)$ and 
a weight function $0 < w(x_1) \le w(x_2) \le \cdots \le w(x_n)$, such that the Hessian determinant 
of the $w$-leading part of $f(Tx)$ is nonzero.
\end{theorem}

\noindent
The example below shows that the above theorem cannot be extended to dimensions larger than three.

\begin{example} \label{weightcounter}
Let $n \ge 4$ and
$$
f := x_1 x_2 + t x_1 x_2^2 + (x_2 + x_1 x_3)^3 + x_1^4 (1 + x_4) +
     (x_5^7 + \cdots + x_n^{n+2})
$$
Then $\det \hess f =  t g$, where 
$$
g = -\tfrac1{450} (n+1)!\, (n+2)!\, x_1^9(x_2 + x_1 x_3) x_5^5 \cdots x_n^n
\in \Z[x] \setminus \{0\}
$$
In particular,
$$
\det \hess (f|_{t=0}) = 0 \ne g = \det \hess (f|_{t=1})
$$
Consequently, for each 
$T \in \GL_n(\C)$ and each $(w(x_1),w(x_2),\ldots,w(x_n)) \in \R^n \setminus \{0\}^n$, the 
$w$-leading part of $f(Tx)|_{t=0}$ has Hessian determinant zero. We will show in section \ref{second} 
that the same holds for $f(Tx)|_{t=1}$, although its Hessian determinant is nonzero.
Hence the condition $n \le 3$ in theorem \ref{weight} is necessary. 

This example was inspired by formula (9) in \cite[Th.\@ 3.5]{MR2095579}, and $f|_{t=0}$ is of the form
of this formula in dimension $n = 4$.
\end{example}

\noindent
As announced, we will use results of \cite{MR2095579} in our proof. These results are used in the proof of 
theorem \ref{weight}, and are as follows.

\begin{theorem} \label{zerohess}
Let $K$ be a field of characteristic zero. Suppose that $h \in K[x] = K[x_1,x_2,\ldots,x_n]$
has no terms of degree less than two and that $\det \hess h = 0$. If $n \le 3$, then there exists a 
$T \in \GL_n(K)$ such that all entries below the anti-diagonal of the Hessian of $h(Tx)$ are zero.
\begin{enumerate}

\item[i)] If $n = 2$, then $h \in K[l_1]$ for some linear form $l_1 \in K[x]$.

\item[ii)] If $n = 3$, then either $h \in K[l_1,l_2]$ for some linear forms 
$l_1, l_2 \in K[x]$, or $h = g_1(l_1) x_1 + g_2(l_1) x_2 + g_3(l_1)x_3$ for some
linear form $l_1 \in K[x]$ and polynomials $g_1, g_2, g_3 \in K[x_1]$. 
Furthermore, the leading homogeneous part of $h$ is of the form $l_1^{\deg h - 1} l_4$ for 
some linear form $l_4 \in K[x]$ in the latter case.

\end{enumerate}
\end{theorem}

\begin{proof}
Since the claims of theorem \ref{zerohess} are void when $n = 1$, the cases $n = 2$ and $n = 3$ remain. 
Write $l = (l_1, l_2, \ldots, l_n)$.
\begin{enumerate}

\item[i)] Assume that $n = 2$. \cite[Th.\@ 3.1]{MR2095579} tells us that there exists a 
$T \in \GL_n(K)$ such that $h(Tx) \in K[x_1]$. Hence $T$ is as claimed and we can take 
$l = T^{-1} x$ to get $h \in K[l_1]$.

\item[ii)] Assume that $n = 3$. \cite[Th.\@ 3.3]{MR2095579} tells us that there exists a 
$T \in \GL_n(K)$ such that $h(Tx) \in K[x_1,x_2]$ or $h(Tx)$ is of the form 
$a_1(x_1) + a_2(x_1) x_2 + a_2(x_3) x_3$ for polynomials $a_1, a_2, a_3$. Again, $T$ is as 
claimed and we can take $l = T^{-1} x$. Then $h = a_1(l_1) + a_2(l_1) l_2 + a_3(l_1)l_3$ in case 
$h \notin K[l_1,l_2]$, which we assume from now on.

Since $h$, $a_2(l_1) l_2$ and $a_3(l_1)l_3$ have 
no constant terms, neither has $a_1(l_1)$, and we see that $l_1 \mid a_1(l_1)$. Consequently,
$h = b_1(l_1)l_1 + b_2(l_1) l_2 + b_3(l_1)l_3$ for polynomials $b_1, b_2, b_3$.
But each of the linear forms $l_1, l_2, l_3$ is a linear combination of $x_1, x_2, x_3$, 
whence $h = g_1(l_1) x_1 + g_2(l_1) x_2 + g_3(l_1)x_3$ for polynomials $g_1, g_2, g_3$.

If $c_1, c_2, c_3$ are the coefficients of degree $\deg h - 1$ of $g_1, g_2, g_3$, respectively,
then the leading homogeneous part of $h$ is equal to $l_1^{\deg h-1} (c_1 x_1 + c_2 x_2 + c_3 x_3)$,
which is as claimed with $l_4 = c_1 x_1 + c_2 x_2 + c_3 x_3$. \qedhere

\end{enumerate}
\end{proof}

\begin{proof}[Proof of theorem \ref{dillenext}] 
There is nothing to be proved when $\det \hess f \notin K$, so let $\det \hess f \in K$. 
Notice that the Jacobian conjecture is trivially satisfied when $n = 1$ or $d \le 2$. 
Since there is nothing additionally to be proved in both cases, we may assume that 
$2 \le n \le 3 \le d$.
\begin{enumerate}

\item[i)]
We first show that a $T$ as given exists. If $\det \hess f = 0$, then the existence of $T$ follows
from theorem \ref{zerohess}. Hence suppose that $\det \hess f \in K^{*}$. By theorem \ref{weight},
there exist a $T \in \GL_n(K)$ and a weight function $0 < w(x_1) \le w(x_2) \le \cdots \le w(x_n)$ 
such that the Hessian determinant of the $w$-leading part of $f(Tx)$ is nonzero. Since 
$\det \hess f \in K^{*}$, the Hessian determinant of the $w$-leading part of $f(Tx)$ is a nonzero
constant as well. On the other hand, all terms of this determinant have weight $n w(f(Tx)) - 
2w(x_1 x_2 \cdots x_n)$, so 
\begin{equation} \label{wdegeq}
n w(f(Tx)) - 2w(x_1 x_2 \cdots x_n) = 0
\end{equation}

By $d \ge 3$, we have $(\hess f(Tx))_{ij} \notin K$ for some $i,j$. 
On account of $0 < w(x_1) \le w(x_2) \le \cdots \le w(x_n)$ and \eqref{wdegeq}, we have 
$$
w\big((\hess f(Tx))_{nn}^{n-1} (\hess f(Tx))_{ij}\big) 
\le n w(f(Tx)) - w(x_i x_j x_n^{n-1} x_n^{n-1}) \le 0
$$
so $(\hess f(Tx))_{nn}^{n-1} = 0$ and $T$ is as claimed when $n = 2$. If $n = 3$, then we have
$$
w\big((\hess f(Tx))_{32}^{n-1} (\hess f(Tx))_{ij}\big) 
\le n w(f(Tx)) - w(x_i x_j x_2^{n-1} x_3^{n-1}) \le 0
$$
so $(\hess f(Tx))_{32}^{n-1} = 0$ as well and $T$ is as claimed when $n = 3$, too.

\item[ii)]
We next show that the quadratic part of $f$ vanishes on the 
last column $Te_n$ of $T$. On account of $n \ge 2$, the lower right corner entry
of $\hess f(Tx)$ vanishes. Hence $f(Tx)$ has no term
which is divisible by $x_n^2$. So every quadratic term of $f(Tx)$ vanishes at the $n$-th
standard basis unit vector $e_n$, and the quadratic part of $f$ vanishes at the last column $T e_n$
of $T$.

\item[iii)]
At last, we show that $\grad f$ satisfies the Jacobian conjecture. 
So assume that $\det \hess f \in K^{*}$. By \eqref{gradT} and $\det T\tp = \det T \ne 0$, 
it suffices to show that $F := \grad f(Tx)$ is invertible. Since $\hess (f(Tx)) = \jac F$,
the invertibility of $F$ follows from lemma \ref{antitri} below. \qedhere

\end{enumerate} 
\end{proof}

\begin{lemma} \label{antitri}
Suppose that $A$ is a commutative $\Q$-algebra and $F \in A[x]^n = A[x_1,x_2,\ldots,x_n]^n$. If all entries
below the anti-diagonal of $\jac F$ are zero and $\det \jac F \in A^{*}$, then $F$ is invertible.
\end{lemma}

\begin{proof}
Since the entries below the anti-diagonal
of $\jac F$ are zero, we see that $F_{n+1-i} \in A[x_1, x_2, \ldots, x_i]$
for each $i$. Hence it follows from $\det \jac F \in A^{*}$, that the entries
on the anti-diagonal of $\jac F$ are nonzero constants. Say that these constants are
$c_1, c_2, \ldots, \allowbreak c_n$ from left to right. Then $F_n - c_1 x_1 \in A$ and 
$F_{n+1-i} - c_i x_i \in A[x_1, x_2, \ldots, x_{i-1}]$ for all $i \ge 2$. By induction on $i$, 
we obtain that $x_i \in A[F_{n+1-i},F_{n+2-i},\ldots,F_n]$ for each $i$, 
so $F$ is invertible. \qedhere
\end{proof}

\noindent
We showed earlier that $f$ in \eqref{dillen4} is a counterexample to theorem \ref{dillenext} with 
$n = 4$. Using some techniques in the above proof, one can show that $f$ is a counterexample
to theorem \ref{weight} with $n = 4$ as well. For that purpose, take any $T \in \GL_4(K)$
and any weight function $w$ such that $0 < w(x_1) \le w(x_2) \le w(x_3) \le w(x_4)$.  

Since the rows of the Hessian of the cubic part of $f$ are independent, there exists
a $j$ such that $(\hess f(Tx))_{nj} \notin K$. Furthermore, $\det \hess f(Tx) \ne 0$
tells us that there exists an $i \ge 3$ and a $k \ge 2$ such that $(\hess f(Tx))_{ik} \ne 0$.
From $0 < w(x_1) \le w(x_2) \le w(x_3) \le w(x_4)$, we deduce that 
$n w(f(Tx)) - 2 w(x_1x_2x_3x_4) \ge n w(f(Tx)) - 2 w(x_jx_kx_ix_n) \ge 
w\big((\hess f(Tx))^2_{nj}(\hess f(Tx))^2_{ik}\big) > 0$, which contradicts \eqref{wdegeq}.

\section{Proofs of theorem \ref{weight} and example \ref{weightcounter}} \label{second}

Let us start with a simple case of theorem \ref{weight}.

\begin{proof}[Proof of the case $n=2$ of theorem \ref{weight}.]
If the leading homogeneous part $\bar{f}$ of $f$ satisfies $\det \hess \bar{f} \ne 0$,
then we can take $w(x_1) = w(x_2) = 1$ and $T = I_2$, where $I_2$ is the identity 
matrix of size $2$. So assume
that $\det \hess \bar{f} = 0$. 

From i) of theorem \ref{zerohess}, it follows that there
exists a $T \in \GL_2(K)$ such that $\bar{f}(Tx) \in K[x_1]$. Since $\bar{f}(Tx)$ is 
homogeneous, we see that its only term is $x_1^d$, where $d = \deg f(Tx)$. 
By \eqref{hessT} on page \pageref{hessT}, we deduce from $\det \hess f \ne 0$ that 
$\det \hess (f(Tx)) \ne 0$. Consequently, $f(Tx)$ has a nonlinear term 
which is divisible by $x_2$ besides the term $x_1^d$.

Take $w(x_1) = 1$. If we take $w(x_2) = 1$, then the $w$-leading part $h$ of 
$f(Tx)$ will have the term $x_1^d$ of $f(Tx)$, but if we take $w(x_2) \ge 1$ large
enough, then $h$ will not have the term $x_1^d$ any more, because $f \notin K[x_1]$.

Now take $w(x_2) \ge 1$ as large as possible, such that $h$ still has the term $x_1^d$ of $f(Tx)$. 
Since $h$ will lose the term $x_1^d$ as soon as $w(x_2)$ increases only a little further,
$h$ must have another term; a term $t$ of which the weight will get larger than that of 
$x_1^d$ as soon as we start increasing $w(x_2)$ any further. 

From $w(t) = w(x_1^d)$ and
$\deg t < d$, it follows that $x_2 \mid t$ and $w(x_2) > w(x_1)$. 
Since $t$ is a term of $f(Tx)$ for which $w(t)$ is maximum 
and $f(Tx)$ has a nonlinear term which is divisible by $x_2$, we deduce that $t$ is not linear 
and that $h$ has no linear terms.

It remains to show that $\det \hess h \ne 0$.
From i) of theorem \ref{zerohess}, it follows that it suffices to show that $h$ cannot
be expressed as a polynomial in one linear form in $x_1$ and $x_2$. Hence suppose that
$h$ can indeed be expressed as a polynomial in one linear form. Since the leading 
homogeneous part of $h$ is a scalar multiple of $x_1^d$, we can take $x_1$ for this linear form. 
But $t \notin K[x_1]$. Contradiction, so $\det \hess h \ne 0$ indeed.
\end{proof}

The case $n = 3$ of theorem \ref{weight} is more complicated. Let us first give a situation where
we can do similar things as in the proof of the case $n=2$ of theorem \ref{weight}, 
to obtain that theorem \ref{weight} holds in that situation.

\begin{lemma} \label{weight2}
Let $n = 3$ and assume that $f \in K[x]$ satisfies $\det \hess f \ne 0$. 
Let $T \in \GL_3(K)$ and let $w$ be a weight function for which 
$0 < w(x_1) \le w(x_2) \le w(x_3)$. 

Suppose that the $w$-leading part $h$ of $f(Tx)$ is contained in $K[x_1,x_2]$, but is
not of the form $g_1(l_3) x_1 + g_2(l_3) x_2$ for any linear form $l_3 \in K[x_1,x_2]$
and any polynomials $g_1, g_2 \in K[x_1]$. 

Then theorem \ref{weight} holds for $f$. More precisely, we only need to adjust $w(x_3)$
by increasing it a certain amount in order to get $T$ and $w$ as in theorem \ref{weight} 
for this particular $f$.
\end{lemma}

\begin{proof}
Take $w'(x_1) = w(x_1)$ and $w'(x_2) = w(x_2)$. By \eqref{hessT} on page \pageref{hessT}, 
we deduce from $\det \hess f \ne 0$ that $\det \hess (f(Tx)) \ne 0$. Hence $f(Tx)$ has a nonlinear term
which is divisible by $x_3$. If we take $w'(x_3) = w(x_3)$, then the $w'$-leading part $h'$ of $f(Tx)$ 
will be just $h \in K[x_1,x_2]$. But if we take $w'(x_3)$ large enough, then $h'$ will not have
any term of $h \in K[x_1,x_2]$, because the value of $w'$ at any term which is divisible by $x_3$ will be 
larger than $w'(h)$. 

Now take $w'(x_3) \ge w(x_3)$ as large as possible, such that $h'$ still shares a 
term with $h$. 
From $w'(x_1) = w(x_1)$, $w'(x_2) = w(x_2)$, and $h \in K[x_1,x_2]$, we deduce
that all terms of $h$ still have the same weight. Hence $h'$ contains all terms of $h$
(with the same coefficients)
and $w'(h') = w(h)$. Furthermore, $w'(t) = w(t)$ for any term $t$ of $h'$ which is not 
divisible by $x_3$. Hence any such $t$ is also a term of $h$ because of $w'(h') = w(h)$. So
$x_3 \mid h' - h$.

Since $h'$ will lose the terms of $h$ as soon as $w'(x_3)$ increases only a little further,
$h'$ must have a term that $h$ does not have: a term $t'$ such that $w'(t')$ will get larger 
than $w'(h')$ as soon as we start increasing $w'(x_3)$ any further. Just like for any other
term of $h' - h$, we have $x_3 \mid t'$. Since $t'$ is a term of $f(Tx)$ for which $w'(t')$ is maximum 
and $f(Tx)$ has a nonlinear term which is divisible by $x_3$, we deduce that $t'$ is not linear 
and that $h'$ has no linear terms. Furthermore, we can deduce from the existence of $t'$ that 
$w' \ne w$, so $w'(x_3) > w(x_3)$.

So it remains to show that $\det \hess h' \ne 0$. From ii) of theorem \ref{zerohess}, it follows 
that it suffices to show that the following cases do not occur.
\begin{itemize}

\item \emph{$h'$ is of the form $g_1(l_1) x_1 + g_2(l_1) x_2 + g_3(l_1)x_3$ for some
linear form $l_1 \in K[x]$ and some polynomials $g_1,g_2,g_3 \in K[x_1]$.} \\
Since $h'$ contains every term of $h \in K[x_1,x_2]$ and $x_3 \mid h'-h$, we deduce that
$h'|_{x_3=0} = h$. Hence $h = g_1(l_3) x_1 + g_2(l_3) x_2$ for some linear form $l_3 \in 
K[x_1,x_2]$. But we assumed that $h$ does not have this form.

\item \emph{$h' \in K[l_1,l_2]$ for certain linear forms $l_1, l_2 \in K[x]$ and
$w(x_1) = w(x_2)$.} \\
Since $w(x_1) = w(x_2) \le w(x_3)$ and $h \in K[x_1,x_2]$, we see that $h$ is the
leading homogeneous part of $f(Tx)$. As $h'$ contains all terms of $h$, $h$ is also 
the leading homogeneous part of $h'$, and just like for $h'$, we have $h \in K[l_1,l_2]$.
There exists a linear combination of $l_1$ and $l_2$ which does not have $x_3$ as a term,
so we may assume that $l_1 \in K[x_1,x_2]$.
Since $h$ is not of the form $g_1(l_3) x_1 + g_2(l_3) x_2$ for any linear form 
$l_3 \in K[x_1,x_2]$ and any polynomials $g_1, g_2 \in K[x_1]$, we can deduce that 
$l_2 \in K[x_1,x_2]$ as well. This contradicts $x_3 \mid t'$.

\item \emph{$h' \in K[l_1,l_2]$ for certain linear forms $l_1, l_2 \in K[x]$ and
$w(x_1) < w(x_2)$.} \\
We first show that $h \notin K[x_1]$ and that the leading homogeneous part $\bar{h}'$ of $h'$ 
is divisible by $x_1$. Since $h$ is not of the form $g_1(l_3) x_1 + g_2(l_3) x_2$ for any linear form 
$l_3 \in K[x_1,x_2]$, we see that $h \notin K[x_1]$ indeed and that $h \notin K[x_2]$. 
So $h \in K[x_1,x_2]$ has a term $t \in K[x_1,x_2]$ which is divisible by $x_1$. 
From $w'(x_1) = w(x_1)$, $w'(x_2) = w(x_2)$ and $w'(x_3) > w(x_3)$, it follows that 
$w'(x_1) < w'(x_2) < w'(x_3)$. Hence for every term $u \in K[x_2,x_3]$ with the same
degree as $t$, we have $w'(u) > w'(t)$. Since the term $t$ of $h$ is also a term of $h'$,
$x_1 \mid \bar{h}'$ indeed.

Since $x_1 \mid \bar{h}'$, it follows from lemma \ref{xkdiv} below that 
$h' \in K[x_1,l_3]$ for some linear form $l_3 \in K[x_2,x_3]$. From
$h'|_{x_3=0} = h \notin K[x_1]$, we deduce that
$h' \notin K[x_1,x_3]$. So $l_3$ has a term which is divisible by $x_2$. 
$l_3$ has a term which is divisible by $x_3$ as well,
because $x_3 \mid t'$.
As terms of the expansion of $x_1^s l_3^r$, $x_1^s x_2^r$ and $x_1^s x_3^r$ are both terms of $h'$
for some $s \ge 0$ and some $r \ge 1$, because $h' \notin K[x_1]$. 
This contradicts $w'(x_2) < w'(x_3)$. \qedhere

\end{itemize}
\end{proof}

\begin{lemma} \label{xkdiv}
Suppose that $h \in K[l_1,l_2]$ for linear forms $l_1, l_2 \in K[x]$.
If the leading homogeneous part of $h$ is divisible by $x_1$, then $h \in K[x_1,l_3]$
for some linear form $l_3 \in K[x_2,x_3]$.
\end{lemma}

\begin{proof}
The leading homogeneous part $\bar{h}$ of $h$ is contained in $k[l_1,l_2]$ as well.
Take $p \in K[x_1,x_2]$ such that $\bar{h} = p(l_1,l_2)$, and let $l_3 = l_1|_{x_1=0}$
and $l_4 = l_2|_{x_1=0}$. From $x_1 \mid \bar{h} = p(l_1,l_2)$, we can deduce that
$p(l_3,l_4) = 0$. So $l_3$ and $l_4$ are algebraically dependent over $K$.
Since $l_3$ and $l_4$ are linear, they are even linearly dependent over $K$.
So $h \in K[x_1,l_3]$ if $l_3 \ne 0$ and $h \in K[x_1,l_4]$ if $l_3 = 0$.
\end{proof}

Having lemma \ref{weight2}, the strategy will be to reduce to lemma \ref{weight2}, under
the assumption that theorem \ref{weight} does not hold, to obtain a contradiction. To do that,
we use lemma \ref{weight1} below, after which we prove the case $n=3$ of theorem \ref{weight}
under the assumption that lemma \ref{weight1} is satisfied. Finally, we will prove lemma 
\ref{weight1}.

\begin{lemma} \label{weight1}
Let $n = 3$ and assume that $f \in K[x]$ satisfies $\det \hess f \ne 0$. Suppose that
the leading homogeneous part of $f(Tx)$ is of
the form $x_1^{d-1} l_4$ for some $T \in \GL_3(K)$ and some linear form $l_4 \in K[x]$.

Then there exist a $T^{*} \in \GL_n(K)$ and a weight function $w$, such that 
$0 < w(x_1) < w(x_2) = w(x_3)$ and such that the following holds for the
$w$-leading part $h^{*}$ of $f(T^{*}x)$.
\begin{enumerate}

\item[a)] $h^{*}$ is not of the form $g_1(l_3) x_1 + g_2(l_3) x_2 + g_3(l_3)x_3$ for any
linear form $l_3 \in K[x]$ and any polynomials $g_1, g_2, g_3 \in K[x_1]$.

\item[b)] $\det \hess h^{*} \ne 0$ or $h^{*} \in K[x_1,x_2]$.

\end{enumerate}
\end{lemma}

\begin{proof}[Proof of the case $n=3$ of theorem \ref{weight}.]
If the leading homogeneous part $\bar{f}$ of $f$ satisfies $\det \hess \bar{f} \ne 0$,
then we can take $w(x_1) = w(x_2) = w(x_3) = 1$ and $T = I_3$, where $I_3$ is the identity 
matrix of size $3$.

So assume that $\det \hess \bar{f} = 0$. On account of ii) of theorem \ref{zerohess},
$\bar{f} \in K[l_1,l_2]$ for some linear forms $l_1, l_2 \in K[x]$. We can choose
$l_1$ and $l_2$ independent of each other, and such that $l_1$ divides $\bar{f}$ at 
least as many times as any other linear combination of $l_1$ and $l_2$ does.

Take $T \in \GL_3(K)$ such that $l_1(Tx) = x_1$ and $l_2(Tx) = x_2$. Then the leading
homogeneous part $\bar{f}'$ of $f' := f(Tx)$ is contained in $K[x_1,x_2]$, and
$x_1$ divides $f'$ at least as many times as any other linear form in $K[x_1,x_2]$ does.
Let $d = \deg f'$. We distinguish two cases.
\begin{itemize}

\item $x_1^{d-1} \nmid \bar{f}'$. \\
Since $x_1$ divides $\bar{f}'$ at least as many times as any other linear form in $K[x_1,x_2]$ does,
we deduce that $\bar{f}'$ is not of the form $l_3^{d-1} l_4$ for any linear forms $l_3, l_4 \in K[x]$.
On account of ii) of theorem \ref{zerohess}, $\bar{f}'$ is not of the form 
$g_1(l_3) x_1 + g_2(l_3) x_2$ for any linear form $l_3 \in K[x_1,x_2]$
and any polynomials $g_1, g_2 \in K[x_1]$. 

Since $\bar{f}' \in K[x_1,x_2]$, we can apply lemma \ref{weight2}
with weight function $w$ such that $w(x_1) = w(x_2) = w(x_3) = 1$, to obtain
that the claim of theorem \ref{weight} is satisfied for this particular $f$, with $T$ as a above
and a weight function $w$ such that $1 = w(x_1) = w(x_2) < w(x_3)$

\item $x_1^{d-1} \mid \bar{f}'$. \\
Take $T^{*}$, $w$ and $h = h^{*}$ as in lemma \ref{weight1}. If $\det \hess h \ne 0$, then theorem
\ref{weight} is satisfied for this particular $f$, so assume that $\det \hess h = 0$. On account of
b) of lemma \ref{weight1}, $h \in K[x_1,x_2]$. Furthermore, $h$ is not of the form 
$g_1(l_3) x_1 + g_2(l_3) x_2$ for any linear form $l_3 \in K[x_1,x_2]$ and any polynomials 
$g_1, g_2 \in K[x_1]$ because of a) of lemma \ref{weight1}.

It follows from lemma \ref{weight2} that there
exists a weight function $w'$ such that $w'(x_1) = w(x_1)$, $w'(x_2) = w(x_2)$ and
$w'(x_3) > w(x_3)$, and such that the $w'$-leading part $h'$ of $f(T^{*}x)$ satisfies 
$\det \hess h' \ne 0$. So theorem \ref{weight} is satisfied for this particular $f$. \qedhere

\end{itemize}
\end{proof}

\begin{proof}[Proof of lemma \ref{weight1}.]
On account of \eqref{hessT} on page \pageref{hessT}, $\det \hess (f(Tx)) \ne 0$ because $\det \hess f \ne 0$.
Since $\det \hess (f(Tx)) \ne 0$, the trailing principal minor matrix of size $2$ of $\hess (f(Tx))$
is not the zero matrix. It follows that $f(Tx)$ has a term which is divisible by 
$x_2^2$, $x_2x_3$, or $x_3^2$.

Take $w(x_2) = w(x_3) = 1$. If $w(x_1) = 1$ as well, then the $w$-leading part $h$ of $f(Tx)$
is just the leading homogeneous part of $f(Tx)$, which is $x_1^{d-1} l_4$ and hence does not 
have a term which is divisible by $x_2^2$, $x_2x_3$ or $x_3^2$. If $w(x_1) > 0$ is small enough, 
then the value of $w$ at terms which are divisible by $x_2^2$, $x_2x_3$ or $x_3^2$ will be
larger than that of terms which are not. 

Now take $w(x_1) \le 1$ as large as possible, such that 
$h$ has a term which is divisible by $x_2^2$, $x_2x_3$ or $x_3^2$. 
Then the part $h'$ of $h$,
consisting of the monomials of $h$ which are divisible by $x_2^2$, $x_2x_3$ or $x_3^2$, is nonzero.
Furthermore, $0 < w(x_1) < 1$.
\begin{enumerate}

\item[a)] For the moment, we show the claim of a) for $h$ instead of $h^{*}$.
Since $h$ will lose all terms of $h'$ as soon as $w(x_1)$ increases only a little, 
$h$ must have a term at which the value of $w$ will get larger than that $w(h')$, 
as soon as we start increasing $w(x_1)$. In particular, $h - h' \ne 0$. Now let $t$ and $t'$
be arbitrary terms of $h - h'$ and $h'$, respectively. From $w(t) = w(t')$ and 
$0 < w(x_1) < w(x_2) = w(x_3)$, we can deduce that $\deg t > \deg t'$ and that $x_1^2 \mid t$.
Consequently, the leading homogeneous part $\bar{h}$ of $h$ is divisible by $x_1^2$.

Suppose that $h$ is of the form $g_1(l_3) x_1 + g_2(l_3) x_2 + g_3(l_3)x_3$ for some
linear form $l_3 \in K[x]$ and some polynomials $g_1, g_2, g_3 \in K[x_1]$. Since
$x_1^2 \mid \bar{h}$, it follows from ii) of theorem \ref{zerohess} that we can take
$l_3 = x_1$. This contradicts $h' \ne 0$, so $h$ is not of the form 
$g_1(l_3) x_1 + g_2(l_3) x_2 + g_3(l_3)x_3$ for any linear form $l_3 \in K[x]$ and any
polynomials $g_1, g_2, g_3 \in K[x_1]$.

\item[b)] If $\det \hess h \ne 0$, then b) is satisfied and we can take $h^{*} = h$ and 
$T^{*} = T$ to fulfill a). So assume that $\det \hess h = 0$. Since $h$ is not of the form 
$g_1(l_3) x_1 + g_2(l_3) x_2 + g_3(l_3)x_3$ for any linear form $l_3 \in K[x]$ and any
polynomials $g_1, g_2, g_3 \in K[x_1]$, it follows from ii) of theorem \ref{zerohess}
that $h \in K[l_1,l_2]$ for some linear forms $l_1, l_2 \in K[x]$. By lemma \ref{xkdiv},
we deduce from $x_1^2 \mid \bar{h}$ that we may assume that $l_1 = x_1$ and 
$l_2 \in K[x_2,x_3]$. Furthermore, we can take $l_2$ nonzero. 

So there exists a linear form $l_3 \in K[x_2,x_3]$ such that $l_1, l_2, l_3$ are 
independent. Take $T^{*} \in \GL_3(K)$ such that $l_1(T^{-1}T^{*}x) = x_1$, 
$l_2(T^{-1}T^{*}x) = x_2$ and $l_3(T^{-1}T^{*}x) = x_3$, and let $h^{*} = h(T^{-1}T^{*}x)$. 
Then $h^{*} \in K[x_1,x_2]$ and just like $h$, $h^{*}$ is not of 
the form $g_1(l_5) x_1 + g_2(l_5) x_2 + g_3(l_5)x_3$ for any linear form $l_5 \in K[x]$ 
and any polynomials $g_1, g_2, g_3 \in K[x_1]$. So $h^{*}$ satisfies both a) and b).

Hence it suffices to show that $h^{*}$ is the $w$-leading part of $f(T^{*}x)$.
Since $l_2$ and $l_3$ are linear forms in $x_2$ and $x_3$ and $l_1 = x_1$,
it follows from $w(x_2) = w(x_3)$ that for every term $u$ of $K[x]$,
the value of $w$ at any term of $u(T^{-1}T^{*}x)$ is equal to $w(u)$.
From this, we can deduce that $h^{*}$ is the $w$-leading part of $f(T^{*}x)$, just
like $h$ is the $w$-leading part of $f(Tx)$. \qedhere

\end{enumerate}
\end{proof}

\begin{proof}[Proof of example \ref{weightcounter}.]
Take $T \in \GL_n(\C)$ and define $l_i = T_i x$ for each $i \le n$. Then
$$
f(Tx) = l_1 l_2 + t l_1 l_2^2 + (l_2 + l_1 l_3)^3 + l_1^4 (1 + l_4) +
        (l_5^7 + \cdots + l_n^{n+2})
$$
Let $w$ be a weight function, such that the Hessian determinant of the $w$-leading part of 
$f(Tx)|_{t=1}$ is nonzero. We shall show that $w(x_i) = 0$ for all $i \le n$. 

If $w(l_1 l_2^2) < w\big(f(Tx)|_{t=1}\big)$, then the $w$-leading part of $f(Tx)|_{t=1}$
is the same as that of $f(Tx)|_{t=0}$, which contradicts that its Hessian determinant is
nonzero. Thus
\begin{equation} \label{x1x2x2}
w (l_1 l_2^2) \ge w \big(f(Tx)|_{t=1}\big)
\end{equation}
We distinguish five cases.
\begin{itemize}
 
\item $0 > w (l_2)$. \\
Then $w(l_1 l_2) > w(l_1 l_2^2)$. Since $l_1 l_2$ is the quadratic part of 
$f(Tx)|_{t=1}$, we have a contradiction with (\ref{x1x2x2}).

\item $0 \le w(l_2) > w(l_1)$. \\
Then $w(l_2^3) > w(l_1 l_2^2)$. Since $l_1 l_2^2 + l_2^3$ is the cubic 
part of $f(Tx)|_{t=1}$, we have a contradiction with (\ref{x1x2x2}).

\item $0 \le w(l_2) \le w(l_1)$ and $w(l_3) > 0$. \\
Then $w(l_1^3l_3^3) > w(l_1 l_2^2)$. Since $l_1^3 l_3^3$ is the 
part of degree six of $f(Tx)|_{t=1}$, we have a contradiction with (\ref{x1x2x2}).

\item $0 \le w(l_2) \le w(l_1) > w(l_3) \le 0$. \\
Since $\tilde{f} := f(Tx)|_{t=1} - \big((l_2 + l_1 l_3)^3 - l_2^3\big) 
\in \C[l_1,l_2,l_4,l_5,\ldots,l_n]$, it follows that $\tilde{f}(T^{-1}x) \in 
\C[x_1,x_2,x_4,x_5,\ldots,x_n]$. Hence $\det \hess \tilde{f} = 0$ 
on account of \eqref{hessT} on page \pageref{hessT}. Furthermore,
$l_1^4$ is the part of degree four of $\tilde{f}$, and
$$
w(l_1^4) > w\big(3l_1 l_2^2 l_3 + 3 l_1^2 l_2 l_3^2 + l_1^3 l_3^3\big) 
               = w\big((l_2 + l_1 l_3)^3 - l_2^3\big)
$$
So the $w$-leading parts of $\tilde{f}$ and $f(Tx)|_{t=1}$ are equal, and
their Hessian determinants are zero because $\det \hess \tilde{f} = 0$. 
Contradiction.

\item $0 \le w(l_2) \le w(l_1) \le w(l_3) \le 0$. \\
Then $w(l_i) = 0$ for each $i \le 3$. If $w(l_4) > 0$, then 
$w(l_1^4 l_4) > w\big((l_2 + l_1 l_3)^3 - l_2^3\big)$ and just as in the case above,
we get a contradiction because the $w$-leading parts of $\tilde{f}$ and $f(Tx)|_{t=1}$ are 
equal and $\det \hess \tilde{f} = 0$. Thus $w(l_4) \le 0$ and similarly, 
$w(l_i) \le 0$ for all $i \ge 5$.

Thus $w(l_i) \le 0$ for all $i$, and consequently $w(x_i) \le 0$ for all $i$ as well.
Since $l_1 l_2$ is the quadratic part of $f(Tx)|_{t=1}$ and $w(l_1 l_2) = 0$, we have
$w(f(Tx)|_{t=1}) = 0$ as well. Thus if there exists an $i$ such that $w(x_i) < 0$, 
then we do not have $x_i$ in the $w$-leading part of $f(Tx)|_{t=1}$. So $w(x_i) = 0$ 
for all $i$, as desired. \qedhere

\end{itemize}
\end{proof}

\section{Anisotropic polynomials} \label{last}

The last claim of the main theorem, theorem \ref{dillenext}, is that the quadratic part
of $f$ is so-called {\em isotropic over $K$} in case $2 \le n \le 3 \le d$. The opposite of isotropic 
is anisotropic. Below, the definition of anisotropic is generalized somewhat.

\begin{definition}
Let $K$ be a field of characteristic zero and $f \in K[x] = K[x_1,x_2,\ldots,x_n]$. 
We say that $f$ is {\em anisotropic over $K$ at $\lambda \in K^n$} if the quadratic part
of $f(x+\lambda)$ is anisotropic over $K$, i.e.\@ does not vanish anywhere at $K^n \setminus \{0\}^n$, 
or equivalently, $\mu\tp H \mu \ne 0$ for all $\mu \in K^n \setminus \{0\}^n$, where
\begin{equation} \label{aneq}
H = \big(\hess(f|_{x=x+\lambda})\big)\big|_{x=0} 
  = \big((\hess f)|_{x=x+\lambda}\big)\big|_{x=0} = (\hess f)|_{x=\lambda}
\end{equation}
\end{definition}

\noindent
In the following theorem, the cases $n \le 3$ and $K = \R$ are distinguished. The first case
follows from our techniques, while the second case follows from the result by Pogorelov 
in \cite{MR0319126}, which was mentioned in the introduction. Let $\GO_n(K)$ denote the group
of orthogonal matrices of size $n \times n$ over $K$.

\begin{theorem} \label{definite}
Let $K$ be a field of characteristic zero and $f \in K[x] = K[x_1,x_2,\allowbreak \ldots,x_n]$
such that $\det \hess f \in K^{*}$ and $f$ is anisotropic over $K$ at $\lambda$ for
some $\lambda \in K^n$. If $n \le 3$ or $K = \R$, then $\deg f = 2$.
\end{theorem}

\begin{proof}
By assumption, the quadratic part of $f(x+\lambda)$ does not vanish anywhere at $K^n 
\setminus \{0\}^n$. Hence $\deg f = \deg f(x + \lambda) = 2$ on account the last claim 
of theorem \ref{dillenext} in case $n \le 3$. So assume that $K = \R$. 

Take $\nu \in \R^n \setminus \{0\}^n$ arbitrary. Then there exists a $T_{\nu} \in \GO_n(\R)$ such 
that 
$$
T_{\nu}\tp (\hess f)|_{x=\nu} T_{\nu} = T_{\nu}^{-1} (\hess f)|_{x=\nu} T_{\nu}
$$
is diagonal, see e.g.\@ \cite[Cor.\@ 3.3.1]{MR1923507}. 
Hence all eigenvalues of $(\hess f)|_{x=\nu}$, which are the same as those of
$T_{\nu}\tp (\hess f)|_{x=\nu} T_{\nu}$, are real. Suppose that the eigenvalues of 
$(\hess f)|_{x=\lambda}$ do not have all the same sign. Then 
$T_{\lambda}\tp (\hess f)|_{x=\lambda} T_{\lambda}$ is a diagonal matrix with both positive and negative 
entries, and we can find a $\mu \in \R^n \setminus \{0\}^n$ 
such that $\mu\tp T_{\lambda}\tp (\hess f)|_{x=\lambda} T_{\lambda} \mu = 0$.
This contradicts that $\mu\tp (\hess f)|_{x=\lambda} \mu \ne 0$ for all $\mu \in R^n \setminus \{0\}^n$,
which is satisfied by assumption because of \eqref{aneq}. Hence all eigenvalues of $(\hess f)|_{x=\lambda}$
have the same sign. By replacing $f$ by $-f$ when necessary, we may assume that all eigenvalues of 
$(\hess f)|_{x=\lambda}$ are positive.

From $\det \hess f \in \R^{*}$, it follows that all eigenvalues of $(\hess f)|_{x=\nu}$ are 
positive, because of the continuity of eigenvalues, see e.g. \cite[Th.\@ 3.1.2]{MR1923507}. 
Hence $T_{\nu}\tp (\hess f)|_{x=\nu} T_{\nu}$ is a diagonal matrix without negative 
entries, so $T_{\nu}\tp (\hess f)|_{x=\nu} T_{\nu}$ is positive definite.
Consequently, $(\hess f)|_{x=\nu}$ is positive definite as well. Since the 
main result of \cite{MR0319126} tells us that $\deg f = 2$ in case $\det \hess f \in \R^{*}$
and $(\hess f)|_{x=\nu}$ is positive definite for all $\nu \in \R^n$, the
proof is complete.
\end{proof}

\begin{corollary}
The Jacobian conjecture holds for
gradient maps over the reals whose linear part is the identity map. More precisely, the
corresponding Keller maps are translations.
\end{corollary}

\begin{proof}
Take $K = \R$ and $\lambda = 0$ in theorem \ref{definite}, and notice that the quadratic part of $f$
is $\frac12(x_1^2 + x_2^2 + \cdots + x_n^2)$ in case the linear part of $\grad f$ is the identity map.
\end{proof}

\noindent
A problem for which we do not know the answer, is the following.

\begin{problem}
Does theorem \ref{definite} also hold for all fields $K$ when $n > 3$. 
\end{problem}

\noindent
In the following example, theorem \ref{definite} is applied in a situation where the base
field cannot be embedded into the reals.

\begin{example}
Let $f \in \Q(\I)[x] = \Q(\I)[x_1,x_2,\allowbreak \ldots,x_n]$ with
quadratic part $x_1^2 + 3 x_2^2 + \cdots + (2n-1)x_n^2$.
\begin{enumerate}

\item[i)] If $\det \hess f \in \Q(\I)$ and $n \le 3$, then $\deg f = 2$.

\item[ii)] $x_1^2 + 3 x_2^2 + 5 x_3^2 + 10 x_4^2 = 0$ does not have any nontrivial solution
over $\Q(\I)$. 

\end{enumerate}
\end{example}

\begin{proof}
Since ii) implies that $f$ is anisotropic at the origin, we deduce that
i) follows from ii) and theorem \ref{definite}. Notice that if $n \le 3$ and $\det \hess f$ 
is constant, then $\det \hess f$ is formed by $n$ consecutive digits of $120$, because 
(the constant part of) $\det \hess f$ is totally determined by the quadratic part of $f$.

To prove ii), assume that 
$$
c_1^2 + 3 c_2^2 + 5 c_3^2 + 10 c_4^2 = 0
$$
for certain $c_j \in \Q(\I)$ which are not all zero. Assume without loss of generality that 
$(c_1,c_2,c_3,c_4)$ is a primitive solution of  $x_1^2 + 3 x_2^2 + 5 x_3^2 + 10 x_4^2 = 0$,
i.e.\@ $c_j \in \Z[\I]$ for each $j$ and $\gcd\{c_1,c_2,c_3,c_4\} = 1$ over $\Z[\I]$.
The residue classes modulo $(2 + \I)$ in $\Z[\I]$ can be represented by numbers of $\Z$.
Since $5 = (2-\I)(2 + \I)$, we even have $\Z[\I]/(2 + \I) \cong \F_5$. Let
$\bar{c}_j$ be the element of $\F_5$ which corresponds to the residue class of $c_j$ modulo
$(2 + \I)$ for each $j$.

From $c_1^2 + 3 c_2^2 + 5 c_3^2 + 10 c_4^2 = 0$, we obtain that $\bar{c}_1^2 + 3 \bar{c}_2^2 = \bar{0}$.
But $\bar{c}_1^2 \in \{\bar{0},\bar{1},\bar{4}\}$ and $3\bar{c}_2^2 \in \{\bar{0},\bar{2},\bar{3}\}$,
thus $\bar{c}_1^2 = 3 \bar{c}_2^2 = 0$ and $2 + \I$ divides both
$c_1$ and $c_2$. Similarly, $2 - \I$ divides both
$c_1$ and $c_2$. So $5 \mid c_1$ and $5 \mid c_2$, and we have
$$
c_3^2 + 2 c_4^2 + 5 \left( \frac{c_1}{5} \right)^2 + 15 \left( \frac{c_2}{5} \right)^2 = 0
$$
which gives $5 \mid c_3$ and $5 \mid c_4$ in a similar manner as 
$c_1^2 + 3 c_2^2 + 5 c_3^2 + 10 c_4^2 = 0$ gave $5 \mid c_1$ and $5 \mid c_2$.
This contradicts $\gcd\{c_1,c_2,c_3,c_4\} = 1$.
\end{proof}

\bibliographystyle{consthess3}
\bibliography{consthess3}

\end{document}